\documentclass[]{amsart}

\usepackage{amssymb}           
\usepackage{bm}                

\usepackage{amsmath}           
\usepackage{amsthm}            
\usepackage{amscd}             

\usepackage{ascmac}            
\usepackage{indentfirst}       

\usepackage{graphicx}          

\newtheorem{theorem}{Theorem}[section]
\newtheorem{proposition}[theorem]{Proposition}
\newtheorem{corollary}[theorem]{Corollary}
\newtheorem{lemma}[theorem]{Lemma}

\theoremstyle{definition}
\newtheorem{defi}[theorem]{Definition}
\newtheorem{example}[theorem]{Example}
\newtheorem{remark}[theorem]{Remark}

\newtheorem{problem}[theorem]{Problem}

\newcommand{\R}{\ensuremath{\mathbb{R}}}

\begin{document}



\date{\today}


\title[Roseman moves including triple points]{
Independence of Roseman moves including triple points}


\author{Kengo Kawamura}
\address{Department of Mathematics, Osaka City University, 
3-3-138 Sugimoto-cho, Sumiyoshi-ku, Osaka 558-8585, Japan}
\email{k.kawamura0403@gmail.com}

\author{Kanako Oshiro}
\address{Department of Information and Communication Sciences, Sophia University, 
7-1 Kioi-cho, Chiyoda-ku, Tokyo 102-8554, Japan}
\email{oshirok@sophia.ac.jp}

\author{Kokoro Tanaka}
\address{Department of Mathematics, Tokyo Gakugei University, 
4-1-1 Nukuikita-machi, Koganei-shi, Tokyo 184-8501, Japan}
\email{kotanaka@u-gakugei.ac.jp}



\begin{abstract}
Roseman moves are seven types of local modification for surface-link diagrams in $3$-space 
which generate ambient isotopies of surface-links in $4$-space. 
In this paper, we focus on Roseman moves involving triple points, 
one of which is the famous tetrahedral move, and discuss their independence. 
For each diagram of any surface-link, 
we construct a new diagram of the same surface-link such that 
any sequence of Roseman moves between them must contain moves involving triple points 
(and the numbers of triple points of the two diagrams are the same). 
Moreover, we can find a pair of two diagrams of an $S^2$-knot 
such that any sequence of Roseman moves between them 
must involve at least one tetrahedral move. 
\end{abstract}

\maketitle


\section{Introduction}\label{sec:intro}

A \textit{surface-link} (or a \textit{$\Sigma^2$-link}) is 
a submanifold of $4$-space $\R^4$, 
homeomorphic to a closed surface $\Sigma^2$. 
If it is connected, then it is called a \textit{surface-knot} (or a \textit{$\Sigma^2$-knot}). 
Surface-links are not necessarily assumed to be orientable in this paper. 
Two surface-links are said to be \textit{equivalent} 
if they can be deformed into each other through an isotopy of $\R^4$.

A \textit{diagram} of a surface-link is its image 
via a generic projection from $\R^4$ to $\R^3$, 
equipped with the height information as follows: 
At a neighborhood of each double point, there are 
intersecting two disks such that one is higher than the other 
with respect to the $4$th coordinate dropped by the projection. 
Then the height information is indicated by removing 
the regular neighborhood of the double point in the lower disk 
along the double point curves. 
A diagram 
is regarded as a disjoint union of connected compact orientable surfaces, 
each of which is called a \textit{sheet}, and 
is composed of four kinds of local pictures shown in Figure~\ref{fig:diagram}, 
each of which is the image of a neighborhood of a typical point --- 
a regular point, a \textit{double point}, 
an isolated \textit{triple point} or an isolated \textit{branch point}. 
\begin{figure}[htbp]
\setlength\unitlength{0.70\textwidth}
\begin{center}\begin{picture}(1,0.27)%
\put(0,0.05){\includegraphics[width=\unitlength]{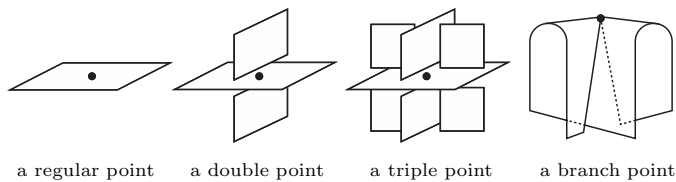}}
\put(0.015,0){\mbox{\scriptsize a regular point}}%
\put(0.275,0){\mbox{\scriptsize a double point}}%
\put(0.545,0){\mbox{\scriptsize a triple point}}%
\put(0.800,0){\mbox{\scriptsize a branch point}}%
\end{picture}\end{center}%
\caption{Local pictures of the projection image}
\label{fig:diagram}
\end{figure}

Two surface-link diagrams are said to be \textit{equivalent} 
if they are related by (ambient isotopies of $\R^3$ and) 
a finite sequence of seven types of Roseman moves, 
shown in Figure~\ref{fig:roseman}, where we omit height information for simplicity 
and 
the symbols \lq\lq B\rq\rq , \lq\lq T\rq\rq\ and \lq\lq D\rq\rq\ 
stand for \lq\lq branch point\rq\rq , \lq\lq triple point\rq\rq\ and 
\lq\lq double point curve\rq\rq\ respectively. 
D. Roseman \cite{Ros-95} proved that two surface-links are equivalent if and only if 
they have equivalent diagrams. 
We refer to \cite{CKS-book, CS-book} for more details on surface-links and 
their diagrams. 

In this paper, we focus on Roseman moves involving triple points, 
that is, the three moves of type $T1$, $T2$ and $BT$, and discuss their independence. 
The move of type $T2$ is also called the \textit{tetrahedral move} 
and closely related to the Zamolodchikov equation, 
which is a higher dimensional analogue of the Yang-Baxter equation 
(see \cite[Chapter 6]{CS-book} for details). 
\begin{figure}[htbp]\begin{center}
\includegraphics[width=1.0\textwidth]{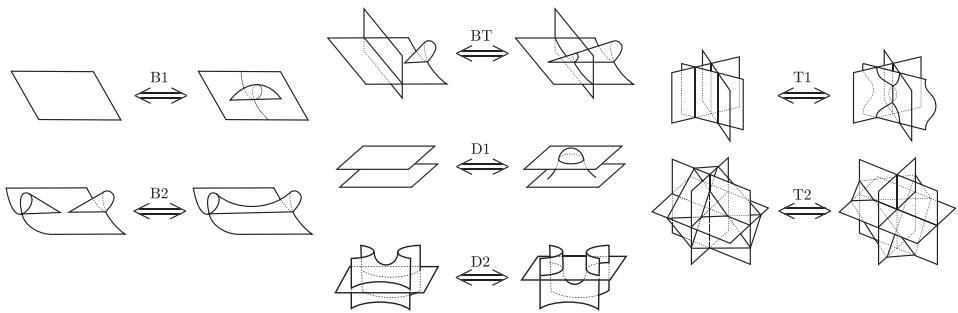}
\caption{Roseman moves}
\label{fig:roseman}
\end{center}\end{figure}

\subsection{Problem}
Independence of Roseman moves has already been well understood as local moves. 
We summarize the known results below. 
The first and third were proved in \cite{Kawamura} 
and the second was proved in \cite{HN, Yashiro}. 
\begin{itemize}
\item
Type $B1$ can be realized by a finite sequence of types $B2$ and $D1$. 
\item 
Type $B2$ can be realized by a finite sequence of types $B1$ and $D2$. 
\item
Any types except $B1$ and $B2$ cannot be realized by the other six types. 
\end{itemize}
However these results do not give us an answer for the following question: 
For two diagrams of a surface-link, what types should be appeared in a sequence 
of Roseman moves between them? 
For example, we can consider the following problem.

\begin{problem}\label{prob:main}\mbox{}
\begin{enumerate}
\item
Are there two diagrams of a surface-link such that 
any sequence of Roseman moves between them must contain moves involving branch points?

\item
Are there two diagrams of a surface-link such that 
any sequence of Roseman moves between them must contain moves involving triple points?
\end{enumerate}\end{problem}

Many studies on Problem~\ref{prob:main}(1) have been made (\cite{OT, Sat-01, TakaseT}). 
On the contrary, there are a few results on Problem~\ref{prob:main}(2), 
for example in \cite{Jab}. 
To make the problem concrete, we define the notion of the \textit{$S$-dependence} 
of diagrams for a subset $S$ of the set consisting of seven types of Roseman moves, 
and formurate our problem. 
\begin{defi}
For a subset $S$ of the set $\{B1, B2, D1, D2, T1, T2, BT\}$ 
consisting of seven types of Roseman moves, 
two diagrams of a surface-link are said to be \textit{$S$-dependent} if 
any sequence of Roseman moves between them contains at least one move in $S$. 
\end{defi}
\begin{problem}\label{prob:main2}
For a subset $S$ of the set $\{B1, B2, D1, D2, T1, T2, BT\}$,  
are there two diagrams of a surface-link such that they are $S$-dependent? 
We note that if we choose $S$ as $\{B1,B2,BT\}$ (resp. $\{T1,T2,BT\}$) 
then the problem is equivalent to Problem~\ref{prob:main}(1) (resp. (2)). 
We also note that if $S' \subset S$ then an $S'$-dependent pair is $S$-dependent by definition. 
Since we are now interested in Roseman moves involving triple points, 
we will take $S$ as a subset of $\{T1, T2, BT\}$ in what follows. 
\end{problem}

\subsection{Results}
M.~Jab{\l}onowski \cite{Jab} 
observed a $\{T1, T2, BT\}$-dependence of surface-link diagrams, and showed that 
there is a pair of two diagrams of the trivial $(S^2 \cup T^2)$-link 
such that the pair is $\{T1,T2,BT\}$-dependent 
and each of two diagrams has no triple points. 
Surface-link diagrams which he constructed are oriented, and 
have multi-components and positive genus, which were crucial conditions in his proof. 
In Section~\ref{sec:T1-T2}, 
we will generalize his result for any surface-link (including  unoriented surface-links, 
surface-knots and $S^2$-links), and prove the following. 
\begin{theorem}\label{thm:T1-T2}
For each diagram $D$ of any surface-link $F$, 
there is a diagram $D'$ of $F$ such that 
the pair of $D$ and $D'$ is $\{T1,T2\}$-dependent and 
the number of triple points of $D$ is equal to that of $D'$. 
\end{theorem}
The first author \cite{Kawamura} observed 
the independence of the moves of types $T1$ and $T2$ as local moves, and showed that 
there is a pair of two diagrams of the trivial $S^2$-link with three (resp. four) components 
such that the pair is $\{T1\}$-dependent (resp. $\{T2\}$-dependent). 
Note that his proof does not work well for surface-links with two (resp. three) or less components. %
In Section~\ref{sec:T2}, 
we will give the first example for the $\{T2\}$-dependence
in the case of an $S^2$-knot, and prove the following. 
\begin{theorem}\label{thm:T2}
There is a pair of two diagrams of an $S^2$-knot such that 
the pair is $\{T2\}$-dependent. 
In other words, any sequence of Roseman moves between them 
must involve at least one tetrahedral move. 
\end{theorem}

Here are some questions for future research.

\begin{problem}
For each diagram $D$ of any surface-link $F$, 
is there a diagram $D'$ of $F$ such that the pair of $D$ and $D'$ is $\{T2\}$-dependent?
\end{problem}

\begin{problem}\label{prob:T1}
Is there a pair of two diagrams of a surface-link with two or less components 
such that the pair is $\{T1\}$-dependent? (See Remark~\ref{rem:T1}.) 
\end{problem}

\begin{problem}
For each diagram $D$ of any surface-link $F$, 
is there a diagram $D'$ of $F$ such that the pair of $D$ and $D'$ is $\{T1\}$-dependent?
\end{problem}

\section{$\{T1, T2\}$-dependent diagram pair}\label{sec:T1-T2}

In this section, we study $\{T1,T2\}$-dependences of equivalent surface-link diagrams using the notion of coloring numbers, and prove Theorem \ref{thm:T1-T2}. 
Throughout this paper, for a surface-link diagram $D$, let $\mathcal{S}_D$ denote the set of all sheets of $D$. Moreover, 
we represent the orientation of a surface-knot diagram by assigning its co-orientation, depicted by an arrow which looks like the symbol \lq\lq $\Uparrow$\rq\rq as in Figure~\ref{fig:coloring_condition}, to each sheet of the diagram. 

\subsection{Coloring of surface-link diagrams}
Let $D$ be an oriented surface-link diagram and $\Omega$ a (non-empty) set with a binary operation $*:\Omega\times\Omega \rightarrow \Omega$.
We say that a map $C:\mathcal{S}_D \rightarrow \Omega$ is an {\it $\Omega$-coloring} of $D$ if it satisfies the {\it coloring condition} $C(s_i)*C(s_j)=C(s_k)$ for each double point curve of $D$, 
where $s_i$, $s_j$ and $s_k$ are three sheets meeting at the double point curve such that 
the co-orientation of $s_j$ points from $s_i$ to $s_k$ as in Figure~\ref{fig:coloring_condition}. 
\begin{figure}[htbp]\begin{center}
\begin{minipage}[]{0.30\hsize}\setlength\unitlength{\hsize}\begin{picture}(1,0.9)%
\put(0,0.05){\includegraphics[width=\unitlength]{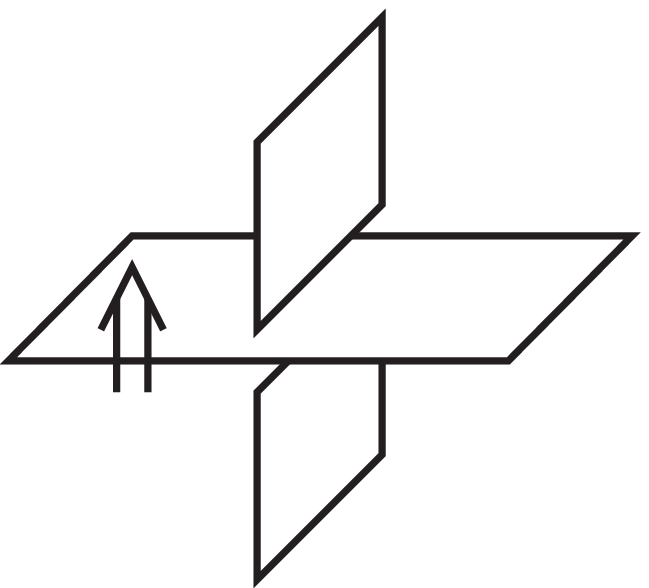}}
\put(0.420,0.250){\mbox{\huge $s_i$}}%
\put(0.605,0.500){\mbox{\huge $s_j$}}%
\put(0.420,0.680){\mbox{\huge $s_k$}}%
\end{picture}\end{minipage}
\qquad{\Large $C(s_i) * C(s_j) = C(s_k)$}
\caption{Coloring condition along a double point curve}
\label{fig:coloring_condition}
\end{center}\end{figure}
Note that there might be no $\Omega$-colorings of $D$ for given $D$ and $\Omega$. 
For an $\Omega$-coloring $C$, the image $C(s)$ of a sheet $s$ is called the {\it color} of $s$.

An $\Omega$-coloring is said to be {\it trivial} if it is a constant map.
An oriented surface-link diagram has a trivial $\Omega$-coloring if and only if there exists an element $a\in\Omega$ such that $a*a=a$.
Let $\mathrm{Col}_\Omega(D)$ denote the set of all $\Omega$-colorings of $D$.
If $\Omega$ is finite, then we can count the number of the elements of $\mathrm{Col}_\Omega(D)$, and we call it the {\it $\Omega$-coloring number} of $D$ and denote it by $\#\mathrm{Col}_\Omega(D)$.
Note that an $\Omega$-coloring number may depend on choice of diagram for an oriented surface-link.

We say that $\Omega$ is {\it involutory} if for any $a,b \in\Omega$, $(a*b)*b=a$ holds. 
If $\Omega$ is involutory, the coloring condition does not depend on the co-orientation of $s_j$, since $C(s_i)*C(s_j)=C(s_k)$ if and only if $C(s_i)=C(s_k)*C(s_j)$. Then the $\Omega$-coloring can be defined for an unoriented diagram of a (possibly non-orientable) surface-link.

\subsection{Quandle and $\{T1, T2\}$-dependence}
The notion of quandles was introduced by Joyce \cite{Joy-82} and Matveev \cite{Mat-82}.
A quandle is a set $Q$ with a binary operation $*:Q \times Q \rightarrow Q$ satisfying 
the following three axioms.
\begin{enumerate}
\item[(Q1)] For any $a \in Q$, $a*a=a$.
\item[(Q2)] For any $a,b \in Q$, there exists a unique $c \in Q$ such that $c*b=a$.
\item[(Q3)] For any $a,b,c \in Q$, $(a*b)*c=(a*c)*(b*c)$.
\end{enumerate}
The axioms (Q1), (Q2) and (Q3) correspond to Reidemeister moves of types I, II and III respectively.

\begin{example}\label{exa:S_4}
Let $S_4$ denote the set $\{0,1,2,3\}$ with the binary operation defined by the following table.
$$
\begin{array}{rc}
S_4:
&
\begin{tabular}{c||c|c|c|c}
  & 0 & 1 & 2 & 3\\ \hline \hline
0 & 0 & 2 & 3 & 1\\ \hline
1 & 3 & 1 & 0 & 2\\ \hline
2 & 1 & 3 & 2 & 0\\ \hline
3 & 2 & 0 & 1 & 3\\       
\end{tabular}
\end{array}
$$
In this table, the $(i+1,j+1)$-entry means $i*j$.
Then $S_4$ satisfies the three axioms 
and we call it the {\it tetrahedron quandle}, 
which will be used in Section~\ref{sec:T2}. 
\end{example}

For a diagram $D$ of an oriented surface-link $F$, 
it is known that the $Q$-coloring number $\#\mathrm{Col}_Q(D)$ is an invariant of $F$ 
for a finite quandle $Q$ (cf.~\cite{CKS-book}). More precisely, it is known that 
\begin{itemize}
\item 
the invariance under Roseman moves of types~$B1$, $B2$ and $BT$ 
comes from the first axiom (Q1), 

\item
the invariance under Roseman moves of types~$D1$ and $D2$ 
comes from the second axiom (Q2), and 

\item
the invariance under Roseman moves of types~$T1$ and $T2$ 
comes from the third axiom (Q3).
\end{itemize}
Hence if we consider the case where an algebra $\Omega$ satisfies quandle axioms (Q1) and (Q2) 
(but not (Q3)), then we have the following. 
\begin{lemma}\label{lem:T1-T2}
Let $\Omega$ be an algebra satisfying quandle axioms $(Q1)$ and $(Q2)$ $($but not $(Q3)$$)$.
If diagrams $D$ and $D'$ of an oriented surface-link are not $\{T1,T2\}$-dependent, 
then $\#\mathrm{Col}_\Omega(D)=\#\mathrm{Col}_\Omega(D')$ holds.
\end{lemma}
\begin{proof}
Since $D$ and $D'$ are not $\{T1, T2\}$-dependent, there exists a finite sequence 
of Roseman moves without $T1$- and $T2$-moves. 
The quandle axiom (Q3) affects only the invariance by $T1$- and $T2$-moves, 
and hence, $\#\mathrm{Col}_\Omega(D)=\#\mathrm{Col}_\Omega(D')$ holds.
\end{proof}

\subsection{Proof of Theorem \ref{thm:T1-T2}}\label{subsec:proof}
Let $X$ be the algebra composed of the set $\{0,1,2\}$ with the binary operation 
$*:X \times X \rightarrow X$ defined by the following table.
$$
\begin{array}{rc}
X:
&
\begin{tabular}{c||c|c|c}
  & 0 & 1 & 2\\ \hline \hline
0 & 0 & 2 & 1\\ \hline
1 & 1 & 1 & 0\\ \hline
2 & 2 & 0 & 2\\
\end{tabular}
\end{array}
$$
In this table, the $(i+1,j+1)$-entry means $i*j$.
The algebra $X$ satisfies the quandle axioms (Q1) and (Q2) but not (Q3).
For example, $(0*1)*2=2*2=2$ and $(0*2)*(1*2)=1*0=1$.
Since the algebra $X$ is involutory, we can apply Lemma \ref{lem:T1-T2} to 
unoriented diagrams of a (possibly non-orientable) surface-link. 
We note that Z.~Cheng and H.~Gao \cite{CG} found the algebra $X$ independently 
in the study of classical knot diagrams. They used $X$-coloring numbers to investigate 
the independence of Reidemeister moves of type III.

\begin{proposition}\label{prop:1}
Let $D_0$ be the trivial diagram of the trivial $S^2$-knot, that is, $D_0$ is the standard $2$-sphere in $\mathbb{R}^3$. 
There exists a diagram $D_1$ of the trivial $S^2$-knot such that the pair of $D_0$ and $D_1$ is $\{T1,T2\}$-dependent, and they have no triple points.
\end{proposition}
\begin{proof}
Let $D_1$ be the $S^2$-knot diagram obtained by spinning the tangle diagram 
in Figure~\ref{fig:trivial_2-knot}. 
\begin{figure}[htbp]\begin{center}
$D_0$: \ \begin{minipage}{0.17\hsize}
\includegraphics[width=\hsize]{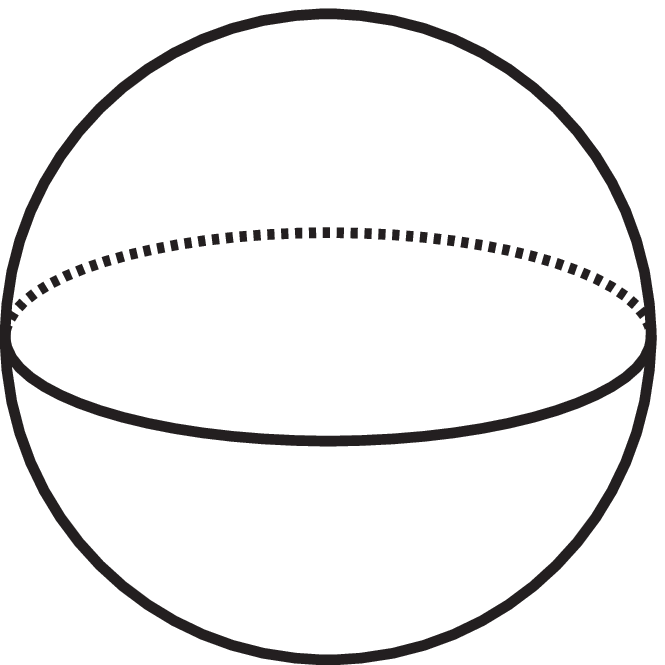}
\end{minipage} \qquad 
$D_1$: \ \begin{minipage}{0.50\hsize}
\includegraphics[width=\hsize]{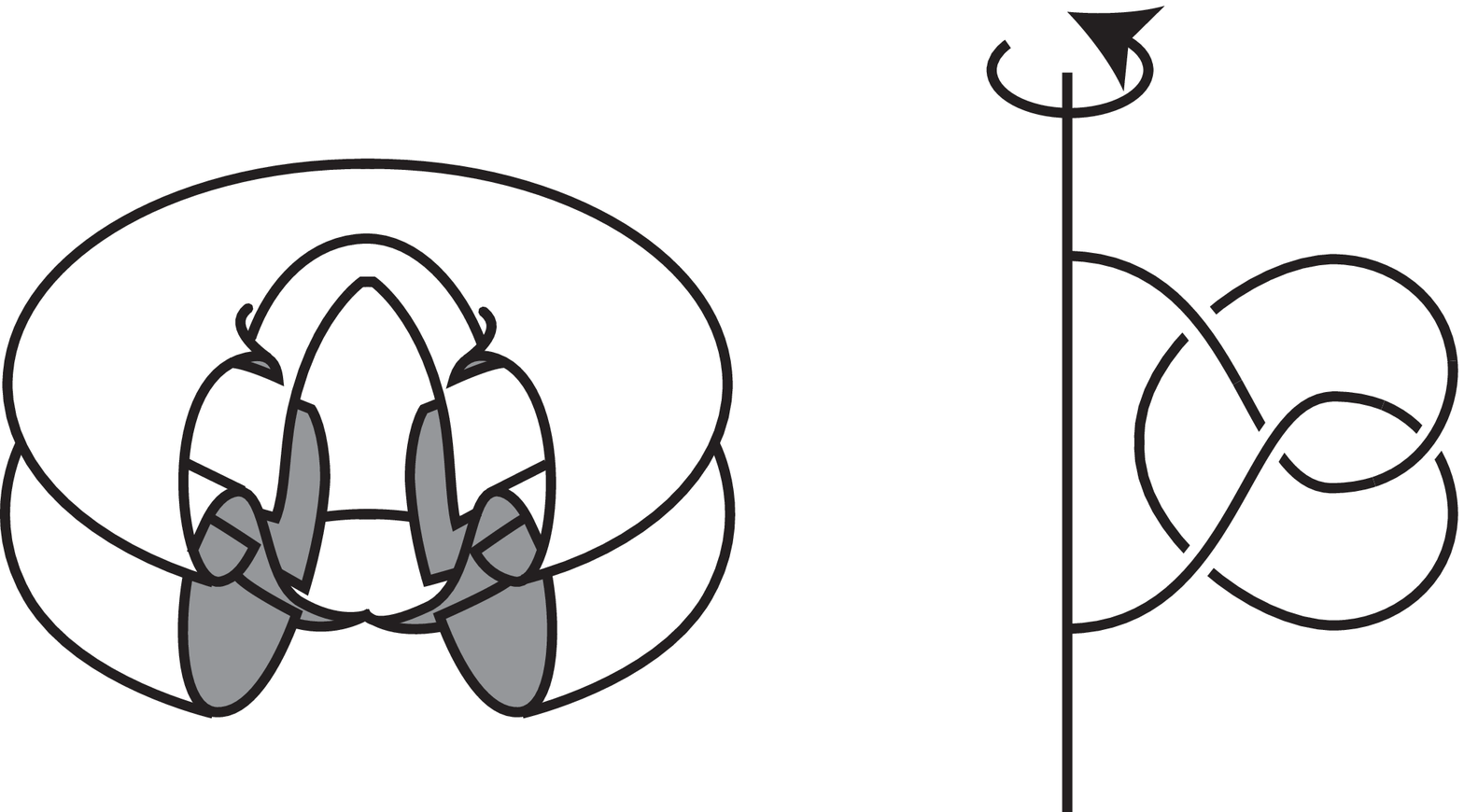}
\end{minipage}
\caption{Two diagrams $D_0$ and $D_1$ of the trivial $S^2$-knot}
\label{fig:trivial_2-knot}
\end{center}\end{figure}
Since the tangle diagram in Figure~\ref{fig:trivial_2-knot} represents the unknotted tangle, 
$D_1$ represents the trivial $S^2$-knot. Moreover, $D_0$ and $D_1$ have no triple points.
Since $D_0$ consists of a single sheet and the algebra $X$ consists of three elements, 
$D_0$ has three trivial $X$-colorings, and hence we have $\#\mathrm{Col}_X(D_0)=3$. 
On the other hand, $D_1$ has three trivial $X$-colorings and four non-trivial ones, 
and hence we have $\#\mathrm{Col}_X(D_1)=7$. 
Thus, we have $\#\mathrm{Col}_X(D_0) \neq \#\mathrm{Col}_X(D_1)$, 
which implies that the pair of $D_0$ and $D_1$ is $\{T1,T2\}$-dependent 
by Lemma \ref{lem:T1-T2}. 
\end{proof}

\begin{remark}\label{rem:T1}
Since there exists a finite sequence of Roseman moves of types $B1$, $B2$ and $T1$ 
between $D_0$ and $D_1$ above, the pair of them is not $\{T2\}$-dependent. 
It is therefore natural to ask the following question: 
Is the pair of $D_0$ and $D_1$ above $\{T1\}$-dependent? (See Problem~\ref{prob:T1}.) 
\end{remark}

\begin{proof}[Proof of Theorem \ref{thm:T1-T2}]
For any diagram $D$ of a surface-link $F$, we set a diagram $D'$ to be a connected sum $D \sharp D_1$ of $D$ and $D_1$ as in Figure \ref{fig:piping}, where $D_1$ is the diagram of the trivial $S^2$-knot in Figure \ref{fig:trivial_2-knot}. 
\begin{figure}[htbp]\begin{center}
$D$: \ \begin{minipage}{0.11\hsize}
\includegraphics[width=\hsize]{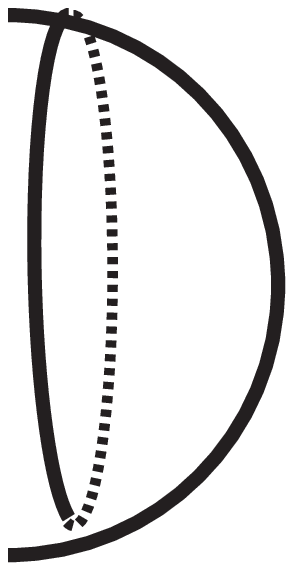}
\end{minipage} {\Huge $\rightsquigarrow$} \quad 
$D'$: \ \begin{minipage}{0.41\hsize}
\includegraphics[width=\hsize]{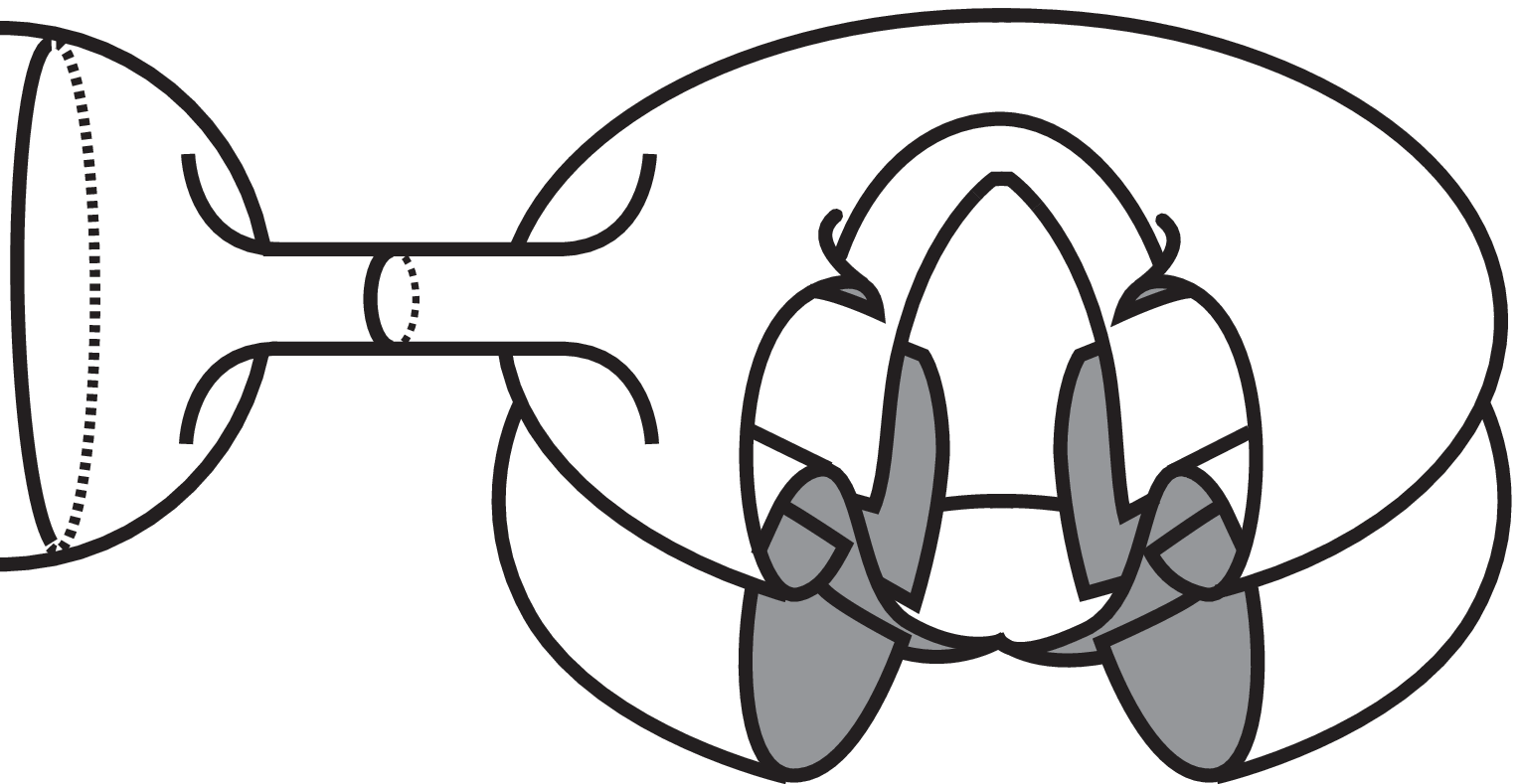}
\end{minipage}
\caption{The diagrams $D$ and $D'=D \sharp D_1$ of the surface-link $F$}
\label{fig:piping}
\end{center}\end{figure}
Then $D'$ is a diagram of $F$ and the number of triple points of $D$ is equal to that of $D'$. 
Let $s_1$ be the sheet of $D_1$ where we have performed the connected sum. 
Then for each element, say $a$, in $X$, there are at least two $X$-colorings of $D_1$ 
such that the sheet $s_1$ receives the color $a$. 
Hence the $X$-coloring number of $D'$ is at least twice of that of $D$, that is, 
$\#\mathrm{Col}_X(D') \geq 2 \times \#\mathrm{Col}_X(D) > \#\mathrm{Col}_X(D) \geq 3$. 
By Lemma \ref{lem:T1-T2}, we conclude that the pair of $D$ and $D'$ is $\{T1,T2\}$-dependent.
\end{proof}

\begin{remark}
Theorem~\ref{thm:T1-T2} is a generalization of Jab{\l}onowski's 
result \cite[Theorem~1.3]{Jab}. 
However we cannot prove that the pair of his surface-link diagrams 
is $\{T1,T2\}$-dependent by using $X$-coloring number.  
\end{remark}

\section{$\{T2\}$-dependent diagram pair}\label{sec:T2}

J.S. Carter et al.~\cite{CESS06} studied 
the number of Reidemeister moves of type III needed for two diagrams of the same classical knot 
using a modification of quandle cocycle invariants. 
In this section, we study $\{T2\}$-dependences of equivalent $S^2$-knot diagrams 
using a similar idea to the one in \cite{CESS06}. 

\subsection{Multi-set $\Phi_\theta (D)$} 
Let $D$ be a diagram of an oriented surface-link and  $\tau$ a triple point of $D$.
For a small neighborhood of $\tau$, the complement of $D$ is divided into eight regions. 
We denote by $R$ one of the eight regions from which all co-orientations 
of the three sheets point outward. 
Let $s_T$, $s_M$ and $s_B$ be the top, middle and bottom sheets of $\tau$ respectively, 
which bounds the region $R$. 
Note that when three sheets form a triple point, they have positions {\it top}, {\it middle} and {\it bottom} with respect to the  height information of the $4$th coordinate dropped by the projection from $\mathbb R^4$ to $\mathbb R^3$.

Let $Q$ be a quandle and $A$ an abelian group. We set a function $\theta: Q^3  \to A$. 
For a $Q$-coloring $C:\mathcal{S}_D \to Q$ of the diagram $D$, 
the (Boltzmann) weight $B_\theta (\tau, C)$ of $\tau$ is defined to be $\varepsilon \, \theta(C(s_B), C(s_M), C(s_T)) \in A$, where $\varepsilon = +1$ if the co-orientations of $s_T$, $s_M$ and $s_B$ in this order matches the orientation of $\mathbb R^3$ and $\varepsilon = -1$ otherwise. 
We denote by $W_\theta (D,C)$ the sum of the weights of all triple points of $D$, that is,  
\[
W_\theta (D,C) = \sum_{\tau\in \{\mbox{triple points of $D$}\} } B_\theta (\tau, C) \in A.
\]

Let $D$ and $D'$ be oriented surface-link 
diagrams such that  $D'$ is obtained from $D$ by a single $T2$-move. For a $Q$-coloring $C$ of $D$, there is a unique $Q$-coloring $C'$ of $D'$ such that these two $Q$-colorings coincide in the complement of the $3$-ball where the $T2$-move is performed. For such pairs of $(D, C)$ and $(D', C')$, we have the following.  
\begin{lemma}\label{lemma:condition(ii)}{\rm (cf. \cite{CJKLS03})}
There exist some $a,b,c,d \in Q$ such that  
\[
\begin{array}{rl}
W_\theta (D,C) -W_\theta (D',C') &=\pm\big[ \theta(a,c,d) - \theta(a*b,c,d)-\theta(a,b,d)\\[2mm]
&\ \ \ \ +\theta(a*c,b*c,d)+\theta(a,b,c)-\theta(a*d,b*d,c*d)\big]
\end{array}
\]
for the pairs of $Q$-colored diagrams $(D, C)$ and $(D', C')$ related by a single $T2$-move. 
\end{lemma}

We denote by $\Phi_\theta (D)$ the multi-set
$\{ W_\theta (D,C) ~|~ C\in {\rm Col}_Q(D)\}$. 
For an unoriented diagram $D$ of an orientable surface-link, 
by considering all possible orientations of $D$, we define the multi-set (of multi-sets)
$\Phi_\theta^{\rm unori} (D)$ by 
\[
\{\Phi_\theta (\vec{D})~|~\vec{D}\in \{\mbox{oriented diagrams representing $D$}\} \}.
\]

\subsection{Quandle $3$-cocycle condition and $\{T2\}$-dependence} 
For a quandle $Q$ and an abelian group $A$, 
a function $\theta: Q^3 \to A$ is called a {\it quandle $3$-cocycle} 
if it satisfies the following conditions.
\begin{itemize}
\item[(i)] For any $a, b\in Q$, $\theta (a,a,b) = 0$ and $\theta(a,b,b)=0$.
\item[(ii)] For any $a,b,c,d \in Q$, 
\[
\begin{array}{l}
 \theta(a,c,d) - \theta(a*b,c,d)-\theta(a,b,d)\\
\hspace{1cm}+\theta(a*c,b*c,d)+\theta(a,b,c)-\theta(a*d,b*d,c*d)=0.
\end{array}
\]
\end{itemize} 
The conditions (i) and (ii) are called the {\it quandle $3$-cocycle condition}s, 
which are obtained from  quandle cohomology theories \cite{CJKLS03}.

Lemma~\ref{lemma:condition(ii)} shows that the quandle cocycle condition (ii) coincides with the difference of the sum of  the weights for two $Q$-colored surface-link 
diagrams  related by a single $T2$-move. This implies that when we set a function $\theta: Q^3 \to A$ to satisfy the quandle cocycle condition (ii), the multi-set $\Phi_\theta (D)$ for a diagram $D$ of an oriented surface-link is unchanged under $T2$-moves. 
Additionally, we can easily check that the quandle $3$-cocycle condition (ii) does not affect the other types of Roseman moves. 
We can similarly see that the quandle $3$-cocycle condition (i) guarantees the invariance of $\Phi_\theta (D)$ under $BT$-moves. 
It is shown in \cite{CJKLS03} that if the function $\theta$ satisfies the quandle $3$-cocycle conditions, the multi-set  $\Phi_\theta (D)$ is independent of choice of diagram, and thus, it is an invariant of oriented surface-links. 
Now, we consider the case where a function $\theta: Q^3 \to A$ 
satisfies the quandle $3$-cocycle condition (i) (but not (ii)).  
\begin{proposition}\label{prop:T2-cocycle} 
For a quandle $Q$ and an abelian group $A$, we set a function $\theta: Q^3 \to A$ satisfying the  quandle $3$-cocycle condition {\rm (i)} {\rm (}but not {\rm (ii)}{\rm )}.
Let $D$ and $D'$ be oriented diagrams which represent the same oriented surface-link. 
If $D$ and $D'$ are not $\{T2\}$-dependent,  
then we have $\Phi_\theta (D) = \Phi_\theta (D')$. 
\end{proposition} 
\begin{proof}
Since $D$ and $D'$ are not $\{T2\}$-dependent, 
there exists a finite sequence of  Roseman moves without $T2$-moves. 
The quandle $3$-cocycle condition (ii) affects  only the invariance by $T2$-moves, 
and hence, $\Phi_\theta (D) = \Phi_\theta (D')$ holds.
\end{proof}

For an unoriented diagram of an orientable surface-link, we have the following. 
\begin{corollary}\label{cor:T2-cocycle}
For a quandle $Q$ and an abelian group $A$, we set a function $\theta:  Q^3 \to A$ 
satisfying the quandle $3$-cocycle condition {\rm (i)} {\rm (}but not {\rm (ii)}{\rm )}.
Let $D$ and $D'$ be unoriented diagrams which represent the same orientable surface-link. 
If $D$ and $D'$ are not $\{T2\}$-dependent,  
then we have $\Phi_\theta ^{\rm unori} (D) = \Phi_\theta^{\rm unori} (D')$. 
\end{corollary}

\subsection{Proof of Theorem~\ref{thm:T2}}
To prove Theorem~\ref{thm:T2}, 
we will construct two concrete $S^2$-knot diagrams by using the deform-spinning method, 
which is reviewed below, defined by R.~A.~Litherland \cite{Litherland79}. 
This is a method of constructing $S^2$-knots from a classical knot such that 
a deformation can be applied during the spinning process.  
Note that the twist-spinning method by E.~C.~Zeeman \cite{Zeeman65} 
and the roll-spinning method by R.~H.~Fox \cite{Fox66} are special cases of 
the deform-spinning method.

For a classical knot $K$, consider a properly embedded arc $K_0$ in the unit  $3$-ball $B^3$ such that $K$ is obtained from $K_0$ by connecting the boundary points by a simple arc in $\partial B^3$.   
Let $f_t : B^3 \to B^3 ~(t \in [0,1])$ be an isotopy of $B^3$ rel $\partial B^3$ such that $f_{1} (K_0) = K_0$. 
Define
\[
(S^4, F) = \bigcup_{t \in [0,1] }(B^3, f_t(K_0))/\sim
\]
where $\sim $ stands for 
\[
\left\{
\begin{array}{ll}
(f_0 ( \mbox{\boldmath $x$}), 0) \sim (f_{1} ( \mbox{\boldmath $x$}), 1 ) &\mbox{ for $ \mbox{\boldmath $x$} \in B^3$,}\\
( \mbox{\boldmath $x$},  t) \sim ( \mbox{\boldmath $x$}, t') &\mbox{ for $ \mbox{\boldmath $x$} \in \partial B^3$ and $t, t' \in [0,1 ]$}.
\end{array}
\right.
\]
Then $F$ is a $2$-sphere embedded in a $4$-sphere $S^4$. Removing a point from $S^4\setminus F$, we have an $S^2$-knot $F$ in $\mathbb R^4$. We call $F$ a {\it deform-spun $S^2$-knot} of $K$. 
Note that  the family $\{f_t(K_0) \}_{t\in [0,1]}$  is a motion picture of tangles in $B^3$ which describes  $F$, and thus, any deform-spun $S^2$-knot can be described by a motion picture of tangles in $B^3$.

Let $F$ be a deform-spun $S^2$-knot. 
Let $p : \mathbb R^4 \to \mathbb R^3$ be the projection induced by the natural projection, say also $p$, from $B^3$ to $B^2$ dropping the 3rd coordinate. 
For each $t \in [0,1]$, the image $p(f_t(K_0))$ equipped with the height information 
is a tangle diagram in $B^2$, 
and the family  $\{p(f_t(K_0)) \}_{t\in [0,1]}$ with the height information is a motion picture of tangle diagrams in $B^2$ which describes  a diagram of $F$. 
Thus we can obtain a diagram of $F$ by a motion picture of tangle diagrams in $B^2$. 
Note that each  Reidemeister move of type III in a motion picture produces a triple point of the corresponding generic projection. Furthermore, for a quandle $Q$, we can also see a $Q$-colored diagram of $F$ by taking a motion picture of $Q$-colored tangle diagrams  in $B^2$.

Let $\tau^3 (3_1)$ denote the deform-spun $S^2$-knot described by the threefold repetition of the motion of tangle diagrams in Figure~\ref{move1}, where the tangle diagrams in Figure~\ref{move1} represent the left-handed trefoil knot  $3_1$. Let $\rho^{\frac{1}{2}} (4_1)$ be the deform-spun $S^2$-knot described by the motion picture of tangle diagrams in Figure~\ref{move2}, where the tangle diagrams in Figure~\ref{move2} represent the figure-eight knot  $4_1$. We note that $\tau^3 (3_1)$ is usually called the {\it $3$-twist-spun $S^2$-knot} of $3_1$ and that $\rho^{\frac{1}{2}} (4_1)$ is usually called the {\it $($half$\text{-})$roll-spun $S^2$-knot} of $4_1$.
\begin{figure}[htbp]
  \begin{center}
\includegraphics[width=0.65\textwidth]{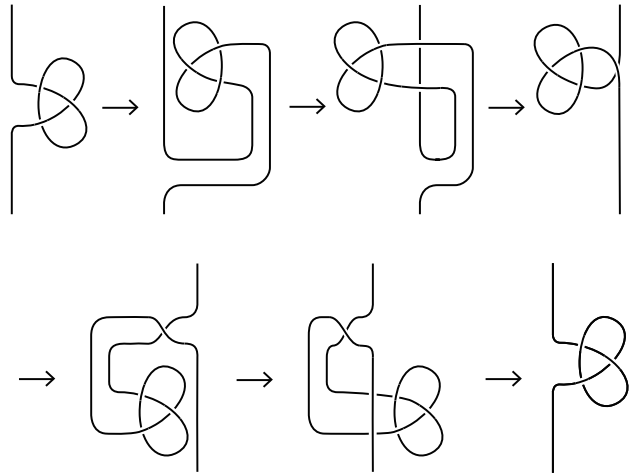}
    \caption{The motion picture of the $1$-twist of $\tau^3 (3_1)$}
    \label{move1}
  \end{center}
\end{figure}

\begin{figure}[htbp]
  \begin{center}
\includegraphics[width=0.65\textwidth]{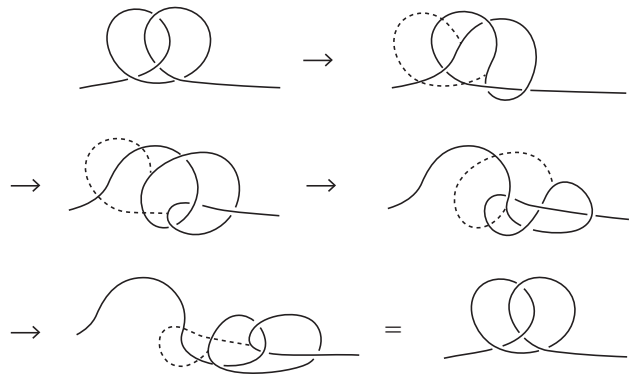}
    \caption{The motion picture of $\rho^{\frac{1}{2}} (4_1)$}
    \label{move2}
  \end{center}
\end{figure}

\textit{Proof of Theorem~\ref{thm:T2}.} 
Let $D$ denote the diagram, of $\tau^3 (3_1)$, 
described by the threefold repetition of the motion of tangle diagrams 
in Figure~\ref{move1}, 
and let $D'$ be the diagram, of $\rho^{\frac{1}{2}} (4_1)$,  
described by the motion picture of tangle diagrams in Figure~\ref{move2}. 
We note that $D$ has eighteen triple points and that $D'$ has twelve triple points. 
Since it is shown in \cite{Kanenobu,Teragaito} that two $S^2$-knots 
$\tau^3 (3_1)$ and $\rho^{\frac{1}{2}} (4_1)$ are equivalent, 
$D$ and $D'$ are also equivalent as $S^2$-knot diagrams. 
In what follows, we focus on proving 
that the pair of $D$ and $D'$ is $\{T2 \}$-dependent.

Let $S_4$ be the tetrahedron quandle given in Example~\ref{exa:S_4} 
and define a map  $\theta: (S_4)^3 \to \mathbb Z$ by $\theta (x,y,z) = (x-y) (y-z).$ 
Note that the map $\theta$ satisfies the quandle $3$-cocycle condition (i) but not (ii). 
To prove this theorem, it is enough to show 
that the multi-set $\Phi_\theta^{\rm unori}(D)$ does not coincide with 
the multi-set $\Phi_\theta^{\rm unori}(D')$ by Corollary~\ref{cor:T2-cocycle}.  
More precisely, for some oriented diagram $\vec{D'}$ of $D'$, 
we show that there exists an $S_4$-coloring $C'$ of $\vec{D'}$ 
such that the weight sum $W_\theta (\vec{D'}, C')$ is non-zero. 
On the other hand, for any oriented diagram $\vec{D}$ of $D$ and 
any $S_4$-coloring $C$ of $\vec{D}$, 
we show that the weight sum $W_\theta (\vec{D}, C)$ is zero, 
that is, $\Phi_\theta^{\rm unori} (D) =\{  0_{16}, 0_{16} \}$, 
where $0_{16}$ is the multi-set composed of $16$ zeros. 

First, consider the multi-set $\Phi_\theta^{\rm unori}(D)$. 
(We refer to \cite{Sat-02} for computation of 
quandle cocycle invariants of twist-spun $S^2$-knots.) 
Set the orientation of $D$ as shown in Figure~\ref{move1-ori} and we denote by $\vec{D}$ the oriented diagram. Note that Figure~\ref{move1-ori} shows the first $1$-twist of $D$. 
We also note that each of the deformations (M1) and (M2) produces three triple points.
For any elements $a, b \in S_4$, Figure~\ref{move1-ori} represents an $S_4$-coloring 
(of the first $1$-twist) 
of $\vec{D}$, that is, the assignment of elements of $S_4$ satisfies the $S_4$-coloring condition. 
Note that when we replace each element, say $x \in S_4$, appeared in Figure~\ref{move1-ori} by $x*a$ (resp.~$(x*a)*a$), the motion of the replaced $S_4$-colored tangle diagrams describes the $S_4$-coloring of the second (resp. third) $1$-twist of $\vec{D}$.  
We also note that the arc colored by $a*b$ 
just before the deformation (M2) in Figure~\ref{move1-ori} 
receives the color $b$ just after the same deformation, 
because $(a*b)*a=b$ holds for any $a,b \in S_4$. 
Since $((x*a)*a)*a=b$ holds  for any $x,a \in S_4$, the last $S_4$-colored tangle diagram of the third $1$-twist of $\vec{D}$ coincides with the first one in Figure~\ref{move1-ori}. 
\begin{figure}[htbp]
  \begin{center}
\includegraphics[width=0.65\textwidth]{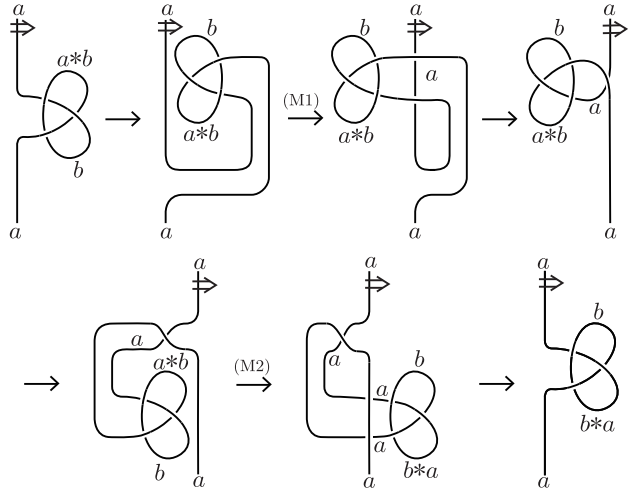}
    \caption{$S_4$-coloring of the $1$-twist of $\tau^3 (3_1)$}
    \label{move1-ori}
  \end{center}
\end{figure}

For the first $1$-twist, the sum of the weights of six triple points appeared in this step is 
\begin{equation}
\begin{aligned}
-\theta(b,b,a*b) - \theta(b,a,b) -\theta (b,a*b ,a)\hspace{2cm} \\ 
+\theta (b,a*b,a)+\theta (a,b,a) + \theta (a*b,a,a), \label{eq:1-twist}
\end{aligned}
\end{equation}
where the first row is obtained from the deformation (M1) in Figure~\ref{move1-ori} and 
the second row is from the deformation (M2) in Figure~\ref{move1-ori}.
It is easy to see that the third term and the fourth term are canceled and the first term and the last term are zero.
Since 
\[
-\theta (b,a,b) + \theta (a,b,a) = -(b-a)(a-b) + (a-b)(b-a) =0,
\]  
the formula (\ref{eq:1-twist}) is equal to $0$.
Similarly, for the second $1$-twist, we also have  
\[
\begin{array}{l}
-\theta(b*a,b*a,b) - \theta(b*a,a,b*a) -\theta (b*a,b ,a)\\[5pt]
+\theta (b*a,b,a)+\theta (a,b*a,a) + \theta (b,a,a)
=0,
\end{array}
\] 
and for the third $1$-twist, we also have 
\[
\begin{array}{l}
-\theta(a*b,a*b,b*a) - \theta(a*b,a,a*b) -\theta (a*b,b*a ,a)\\[5pt]
+\theta (a*b,b*a,a)+\theta (a,a*b,a) + \theta (b*a,a,a)
=0.
\end{array}
\]
Therefore the sum of the weight of all eighteen triple points of $\vec{D}$ is zero. 
It implies $\Phi_\theta (\vec{D}) =\{ 0_{16}\}$.
In the case where we set the reversed orientation for $D$, we can also similarly see that 
the multi-set is $\{0_{16}\}$. Hence we have $\Phi_\theta^{\rm unori}(D)=\{0_{16}, 0_{16}\}$. 

Let us consider the multi-set $\Phi_\theta^{\rm unori}(D')$. 
(We refer to \cite{IS} for computation of 
quandle cocycle invariants of roll-spun $S^2$-knots.) 
Set the orientation and the $S_4$-coloring of $D'$ as shown in Figure~\ref{move2-ori}. 
We denote by $\vec{D}'$ the oriented diagram of $D'$.
\begin{figure}[htbp]
  \begin{center}
\includegraphics[width=0.65\textwidth]{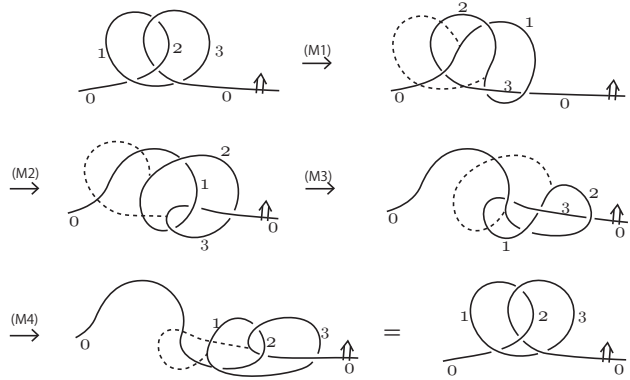}
    \caption{$S_4$-coloring of $\rho^{\frac{1}{2}} (4_1)$}
    \label{move2-ori}
  \end{center}
\end{figure}
 Note that each of the deformations (M1)-(M4) produces three triple points. 
By the direct calculation, we can see that the sum of the weights is 
\[
\begin{array}{l}
\theta(3,0,1)+\theta(1,2,1)-\theta(1,3,1)\\[0.1cm]
+\theta(0,1,2)-\theta(2,1,3)-\theta(2,3,1)\\[0.1cm]
+\theta(3,1,2)+\theta(1,3,2)-\theta(3,0,2)\\[0.1cm]
+\theta(1,0,1)-\theta(3,1,3)-\theta(1,0,2)=12,
\end{array}
\]
where the first, second, third and last rows of the left-hand side are obtained from 
the deformations (M1), (M2), (M3) and (M4) in Figure~\ref{move2-ori}, respectively.  
This implies that $\Phi_\theta (\vec{D'})$ has a non-zero element, and we have 
$\Phi_\theta^{\rm unori}(D') \neq \Phi_\theta^{\rm unori}(D) = \{0_{16}, 0_{16}\}$. 
\hfill \qed

By taking the connected sum of $D$ and $D'$ above 
with the trivial orientable surface-knot diagram, 
we can show the following in a similar way. 

\begin{corollary}
There is a pair of two diagrams of an orientable surface-knot 
such that the pair is $\{T2\}$-dependent. 
In other words, any sequence of Roseman moves between them 
must involve at least one tetrahedral move. 
\end{corollary}


\section*{Acknowledgments}
The second author is partially supported by
Grant-in-Aid for Young Scientists (B) (No.~25800052), 
Japan Society for the Promotion of Science. 
The third author is partially supported by 
Grant-in-Aid for Scientific Research (C) (No.~26400082), 
Japan Society for the Promotion of Science.





\begin{thebibliography}{99}

\bibitem{AG-03}
N. Andruskiewitsch, M. Gra\~{n}a, 
\textit{From racks to pointed Hopf algebras,}
Adv. Math. \textbf{178} (2003), no. 2, 177--243. 

\bibitem{CESS06}
J. S. Carter, M. Elhamdadi, M. Saito, S. Satoh, 
\textit{A lower bound for the number of Reidemeister moves of type III,} 
Topology Appl.  \textbf{153} (2006), no. 15, 2788--2794. 

\bibitem{CG}
Z. Cheng, H. Gao,
\textit{A note on the independence of Reidemeister moves,}
J. Knot Theory Ramifications \textbf{21} (2012), 1220001, 7pp.

\bibitem{CJKLS03}
J.S. Carter, D. Jelsovsky, S. Kamada, L. Langford, M. Saito,  
\textit{Quandle cohomology and state-sum invariants of knotted curves and surfaces,} 
Trans. Amer. Math. Soc.  \textbf{355} (2003),  no. 10, 3947--3989. 

\bibitem{CKS-book}
J. S. Carter, S. Kamada, M. Saito, 
\textit{Surfaces in $4$-space,} 
Encyclopaedia of Mathematical Sciences, \textbf{142},  
Low-Dimensional Topology, III. Springer-Verlag, Berlin, 2004. 

\bibitem{CS-92}
J. S. Carter, M. Saito, 
\textit{Canceling branch points on projections of surfaces in $4$-space,}
Proc. Amer. Math. Soc. \textbf{116} (1992), no. 1, 229--237.

\bibitem{CS-book}
J. S. Carter, M. Saito, 
\textit{Knotted surfaces and their diagrams,} 
Mathematical Surveys and Monographs, \textbf{55},  
American Mathematical Society, Providence, RI, 1998.

\bibitem{FR-92}
R. Fenn, C. Rourke, 
\textit{Racks and links in codimension two,} 
J. Knot Theory Ramifications \textbf{1} (1992), no. 4, 343--406, 

\bibitem{FRS-07}
R. Fenn, C. Rourke, B. Sanderson,
\textit{The rack space,}
Trans. Amer. Math. Soc. \textbf{359} (2007), no. 2, 701--740.

\bibitem{Fox66}
R. H. Fox, \textit{Rolling,}
Bull. Amer. Math. Soc. \textbf{72} (1966), 162--164.

\bibitem{HN}
T. Homma, T. Nagase, 
\textit{On elementary deformations of maps of surfaces into $3$-manifolds. I,} 
Yokohama Math. J.  \textbf{33}  (1985),  no. 1-2, 103--119.

\bibitem{IS}
M. Iwakiri, S. Satoh, 
\textit{Quandle cocycle invariants of roll-spun knots,}
RIMS K\^{o}ky\^{u}roku  \textbf{1766} (2011), 30--37. 

\bibitem{Jab}
M. Jab\l onowski, 
\textit{Knotted surfaces and equivalencies of their diagrams without triple points},
J.\ Knot Theory Ramifications \textbf{21} (2012), 1250019 (6 pages).

\bibitem{Joy-82}
D. Joyce, 
\textit{A classifying invariant of knots, the knot quandle, } 
J. Pure Appl. Algebra \textbf{23} (1982), no. 1, 37--65. 

\bibitem{Kanenobu}
T. Kanenobu,
\textit{Untwisted deform-spun knots: examples of symmetry-spun 2 -knots,}
Transformation groups (Osaka, 1987),  145--167, 
Lecture Notes in Math. \textbf{1375}, Springer, Berlin, 1989. 

\bibitem{Kawamura}
K. Kawamura, 
\textit{On relationship between seven types of Roseman moves},
to appear in Topology Appl.\ 

\bibitem{Litherland79}
R. A. Litherland, 
\textit{Deforming twist-spun knots,} 
Trans. Amer. Math. Soc.  \textbf{250} (1979), 311--331. 

\bibitem{Mat-82}
S. V. Matveev, 
\textit{Distributive groupoids in knot theory,} 
(Russian) Mat. Sb. (N.S.) \textbf{119(161)} (1982), no. 1, 78--88, 

\bibitem{OT}
K. Oshiro, K. Tanaka, 
\textit{On rack colorings for surface-knot diagrams without branch points}, 
to appear in Topology Appl.\ 

\bibitem{Ros-95}
D. Roseman, 
\textit{Reidemeister-type moves for surfaces in four-dimensional space,} 
Knot theory (Warsaw, 1995), 347--380, Banach Center Publ., \textbf{42}, 
Polish Acad. Sci., Warsaw, 1998. 

\bibitem{Sat-01}
S. Satoh, 
\textit{Double decker sets of generic surfaces in $3$-space as homology classes,}
Illinois J. Math. \textbf{45} (2001), no. 3, 823--832. 

\bibitem{Sat-02}
S. Satoh, 
\textit{Surface diagrams of twist-spun $2$-knots,} 
J. Knot Theory Ramifications \textbf{11} (2002), no. 3, 413--430.

\bibitem{TakaseT}
M. Takase, K. Tanaka, 
\textit{Regular-equivalence of $2$-knot diagrams and sphere eversions,}
preprint. 

\bibitem{Teragaito}
M. Teragaito,
\textit{Twisting and Rolling},
RIMS K\^{o}ky\^{u}roku  \textbf{636} (1983), 153--169. 

\bibitem{Yashiro}
T. Yashiro,
\textit{A note on Roseman moves},
Kobe J. Math. \textbf{22} (2005), 31-38.

\bibitem{Zeeman65}
E. C. Zeeman, 
\textit{Twisting spun knots,} 
Trans. Amer. Math. Soc.  \textbf{115} (1965), 471--495.

\end{thebibliography}
\end{document}